\documentclass[12pt,reqno]{amsart}
\usepackage{amsmath,amssymb}
\usepackage{cite}

 \makeatletter
\evensidemargin \oddsidemargin
\makeatother

 \newtheorem{theorem}{Theorem}[section]
\newtheorem{lemma}[theorem]{Lemma}
\newtheorem{proposition}[theorem]{Proposition}

\theoremstyle{definition}

\theoremstyle{remark}
\newtheorem{remark}[theorem]{Remark}
\numberwithin{equation}{section}

 \newtheorem*{theorem*}{Theorem}

\numberwithin{equation}{section}
\numberwithin{equation}{section}

\begin{document}
\title[ Operator mean  inequalities for sector matrices ]{ Operator mean inequalities for sector matrices }

\author[M. Khosravi  ]{Maryam Khosravi$^{1}$}

\author[A. Sheikhhosseini]{  Alemeh
Sheikhhosseini$^2$}

\author[S. Malekinejad]{Somayeh Malekinejad$^{3*}$}

\thanks{*Corresponding author}
 \maketitle

  \begin{center}
  	${}^1$
  	{\small \it Department of Pure Mathematics}\\
  	{\small \it Faculty of Mathematics and Computer}\\
  	{\small  \it Shahid Bahonar University of Kerman,}
  	{\small \it Kerman, Iran} \\
    	{\small \it khosravi$_-$m @uk.ac.ir} \\	\vspace*{5mm}
  	${}^2$
{\small \it Department of Pure Mathematics}\\
  	{\small \it Faculty of Mathematics and Computer}\\
  	{\small  \it Shahid Bahonar University of Kerman,}
  	{\small \it Kerman, Iran} \\
{\small \it sheikhhosseini@uk.ac.ir;hosseini8560@gmail.com} \\\vspace*{5mm}
   ${}^3$
   {\small \it Department of  Mathematics}\\
  	{\small \it Payame Noor University, P.O. Box 19395-3697}\\
  	{\small  \it Tehran, Iran.}\\
    	{\small \it maleki60313@gmail.com} \\
  \end{center}
\begin{abstract}\noindent
In this note, some inequalities involving operator means of sectorial matrices are proved which are generalizations and refinements of previous known results. Among them, let $A$ and $B$ be two accretive matrices with $A,B\in\mathcal{S}_{\theta}$,  $0 < mI \leqslant
A, B \leqslant MI$ for positive real numbers $ M, m, \, \sigma$
 be an operator mean and $\sigma^{*}$ be the adjoint
 mean of $ \sigma.$ If $\sigma^*\leqslant
\sigma_1,\sigma_2\leqslant \sigma$ and $\Phi$ is a
positive unital linear map, then
$$\Phi^{p}\Re(A \sigma_{1} B) \leqslant \sec^{2p}\theta\alpha^{p} \Phi^{p}\Re(A \sigma_{2}
B),$$ where $$ \alpha= \max \left \lbrace K, 4^{1-\frac{2}{p}}K
\right \rbrace,$$ and $ K= \frac{(M+m)^2}{4mM}$ is the
Kantorovich constant.
\end{abstract}
\hspace*{-2.7mm} { KEYWORDS:}{ sector matrix, operator mean, unital positive linear map. }\\
\hspace*{-2.7mm} { MSC[2010]:}{ 47A64,  47A63.}

 \section{introduction and preliminary}
Let $\Bbb M_n$ denote the set of $n\times n$ complex matrices and  $A^*$ denote the conjugate transpose of $A$ for
$A\in\Bbb M_n$.

A Hermitian matrix $A\in \Bbb M_n$  is called positive semidefinite (denoted by $A\geqslant 0$)  if $\langle Ax, x\rangle \geqslant0 $ for all vectors $x \in \mathbb{C}^n$. If in addition $A$ is invertible, then $A$ is called positive definite (denoted by
$A>0$).

 A matrix $A\in M_n$ is called accretive if  its real part $\Re A=\frac{A+A^*}{2}$ is positive definite.

If $A$ is accretive matrix and $W(A)$ is the numerical range of $A$, then
\begin{equation*}
W(A)\subseteq\{z\in \Bbb C:\Re z>0, \vert\Im z\vert\leqslant(\Re z)\tan\theta\}.
\end{equation*}
for some $0\leqslant \theta <\frac{\pi}{2}$. In this case, we will write $A\in {\mathcal S}_{\theta}$.

By definition, it is easily seen that $A\in{\mathcal S}_{\theta}$ is invertible and we have the following well-known inequality \cite{Lin1, Lin2}:
\begin{equation}\label{1}
\Re(A^{-1})\leqslant\Re^{-1}(A)\leqslant \sec^2\theta\ \Re(A^{-1}).
\end{equation}

An operator mean $\sigma$ in the sense of Kubo-Ando \cite{a} is defined by an
operator monotone function $f:(0,\infty)\to(0,\infty)$ with $f(1) = 1$ (briefly we write $f\in\mathbf{m}$) as
$$A\sigma B = A^{1/2}f(A^{-1/2}BA^{-1/2})A^{1/2},$$
for positive invertible matrices $A$ and $B$. The function $f$ is called the representing function of $\sigma$. Recently, Bedrani, Kittaneh and sababheh in \cite{bed1} proved that this definition can be used for accretive matrices too.

For $0\leqslant v\leqslant1$, the $v$-weighted arithmetic,  geometric and Harmonic  means are defined, respectively, by
\begin{align*}
&A\nabla_v B=v A+(1-v)B,\\
&A\sharp_v B=A^{1/2}(A^{-1/2}BA^{-1/2})^{v}A^{1/2},\\
&A!_vB=(vA^{-1}+(1-v)B^{-1})^{-1}.
\end{align*}

 Let $ \sigma $ be an operator mean with representing function $ f. $ The operator
mean with representing function $ f(t^{-1})^{-1} $ is called the
adjoint of $ \sigma $ and denoted by $ \sigma^{*}. $ It follows
from definition that
$$ A \sigma^{*}B=(A^{-1}\sigma B^{-1})^{-1}. $$
For two operator means $\sigma_f$ and $\sigma_g$ with representing functions $f$ and $g$, we write $\sigma\leqslant\tau$ if $A\sigma B\leqslant A\tau B$ for every two positive matrices $A$ and $B$ or equivalently if $f(t)\leqslant g(t)$ for all positive $t\in\mathbb{R}$.

A linear map $\Phi:M_n\rightarrow M_n$ is called positive if
$\Phi(A)\geq0$ whenever $A\geq0$. If $\Phi(I)=I$, where $I$ denoted the identity matrix, then we say that $\Phi$ is unital.

Khosravi et al. in \cite{k17} prove that if $0<mI\leqslant A, B\leqslant MI$, then for every positive unital linear map $\Phi$ and operator means $\sigma_1$ and $\sigma_2$ between some $\sigma$ and $\sigma^*$,
\begin{equation}\label{2.5}
\Phi^p(A\sigma B)\leqslant \alpha ^p \Phi^p(A\sigma^* B),
\end{equation}
where $p>0$, $\alpha=\max\{K,4^{1-\frac{2}{p}}K\}$ and $K=\frac{(M+m)^2}{4mM}$ is the Kantorovich constant.

In addition, they showed that if $\sigma_1$ and $\sigma_2$ between $\nabla_{\nu}$ and $!_{\nu}$, then
\begin{equation}\label{13}
f\left(\Phi(A))\sigma_1f(\Phi( B)\right)\leqslant Kf\left(\Phi(A\sigma_2 B)\right),
\end{equation}
for every nonzero operator monotone function $f$ on $[0,\infty)$.

The main purpose of this article is to present the sectorial version of these inequalities.

Throughout, $K$ considered as the Kantorovich constant $\frac{(M+m)^2}{4mM}$.

\section{main results}
We start this section with some lemmas which are necessary for proving our main results.

\begin{lemma}\label{2}\cite{bed1}
Let $ A, B \in {\mathcal S}_{\theta}$. Then $A\sigma B \in {\mathcal S}_{\theta}$ and
\begin{equation*}
\Re A\sigma \Re B\leqslant \Re (A\sigma B)\leqslant\sec^2\theta(\Re A\sigma \Re B).
\end{equation*}
\end{lemma}
Also we have the
famous Choi's inequality:
\begin{lemma}\label{4}\cite{pec}
Let $ A \in B(H)$ be positive. Then, for every positive unital linear map $\Phi$,
\begin{align*}
\Phi^{-1}(A)\leqslant\Phi(A^{-1}).
\end{align*}
\end{lemma}
By these lemmas, we can prove the next proposition.
\begin{proposition}\label{pro1}
Let $A,B\in {\mathcal S}_{\theta}$ for some $0\leqslant\theta\leqslant\pi/2$ such that $0<mI\leqslant \Re A,\Re B\leqslant MI$. If $\sigma$ is an arbitrary mean and  $\sigma_1,\sigma_2$ are two operator means between $\sigma$ and $\sigma^*$, then
\begin{equation}\label{m1}
\cos^2\theta\Phi \Re(A\sigma_1 B)+mM\Phi^{-1}\Re(A\sigma_2B)\leqslant (M+m)I
\end{equation}
for every  positive unital linear map $\Phi$.
\end{proposition}

\begin{proof}
From $ 0 < mI \leqslant \Re A \leqslant MI$, it follows that $$(MI-\Re A)(\frac{1}{m}I-\Re^{-1}A)\geq0,$$ which is equivalent to $$\Re A + Mm\Re^{-1}(A) \leqslant (M+m)I.$$ Similarly $$\Re B + Mm\Re^{-1}(B) \leqslant (M+m)I.$$
Since every operator mean is subadditive and monotone, we have
\begin{align*}
\Re A \sigma \Re B+ mM(\Re^{-1}(A)\sigma \Re^{-1}(B)) &\leqslant (\Re A + Mm\Re^{-1}(A))\sigma (\Re B + Mm\Re^{-1}(B) )\\
& \leqslant (M+m)I\sigma(M+m)I\\
&= (M+m)I
\end{align*}
Thus, for a positive unital linear map $ \Phi, $ we obtain
\begin{align*}\label{a1}
\Phi(\Re A \sigma \Re B)+ mM\Phi(\Re^{-1}(A)\sigma \Re^{-1}(B)) \leqslant (M+m)I.
\end{align*}
 Let $\sigma^*\leqslant\sigma_1,\sigma_2\leqslant \sigma$. Using  Lemma \ref{2} and Choi's inequality, we get
\begin{align*}
\cos^2\theta\Phi \Re(A\sigma_1 B)+mM\Phi^{-1}\Re(A\sigma_2B)&\leqslant
\Phi(\Re A \sigma_1 \Re B)+ mM\Phi^{-1}(\Re A\sigma_2 \Re B)\\
&\leqslant
\Phi(\Re A \sigma \Re B)+ mM\Phi^{-1}(\Re A\sigma^* \Re B)\\
 &\leqslant \Phi(\Re A \sigma \Re B)+ mM\Phi((\Re A\sigma^{*} \Re B)^{-1} )\\
&=\Phi(\Re A \sigma \Re B)+ mM\Phi(\Re^{-1}A \sigma \Re^{-1}B )\\
& \leqslant (M+m)I,
\end{align*}
and this complete the proof.
\end{proof}
\begin{lemma}\label{2.1}\cite{bh}
Let $A,B\geq0$. Then
\begin{equation*}
\|AB\|\leqslant\frac{1}{4}\|A+B\|^2.
\end{equation*}
\end{lemma}
\begin{lemma}\cite{a}\label{l2}
For each  $ A, B  > 0 $  and $ p > 1, $
$$ \Vert A^{p}+ B^{p}\Vert \leqslant\Vert(A+B)^{p}\Vert.  $$
\end{lemma}
Now we are ready to prove our main results which is an extension of \eqref{2.5} for accretive matrices.
\begin{theorem}\label{t2}
Let $A,B\in {\mathcal S}_{\theta}$ such that $0<mI\leqslant \Re A,\Re B\leqslant MI$ and $\Phi$ be a positive unital linear map. If $\sigma$ is an operator mean and $\sigma_1,\sigma_2$  are two operator means between $\sigma$ and $\sigma^*$, then
\begin{align*}
&\Re^p\Phi(A\sigma_1 B)\leqslant\sec^{2p}\theta\ K ^p \Re ^p\Phi(A\sigma_2 B),&(p\leqslant2)\\
&\Re^p\Phi(A\sigma_1 B)\leqslant\sec^{2p}\theta 4^{p-2} K ^p \Re ^p\Phi(A\sigma_2 B).&(p\geq2)
\end{align*}
\end{theorem}
\begin{proof}
For the first inequality, by Lowner-Heinz inequality, it is enough to prove the inequality for $p=2$.

Applying  Lemma \ref{2.1} and inequality \eqref{m1},  we have
\begin{align*}
\Vert\cos^2\theta \Phi\Re(A \sigma_1 B) &\Phi^{-1}\Re(A\sigma_2 B) \Vert \\&\leqslant  \frac{1}{4mM} \Vert \cos^2\theta\Phi \Re(A\sigma_1 B)+mM\Phi^{-1}\Re(A\sigma_2B)\Vert^{2}\\
& \leqslant \frac{(M+m)^{2}}{4mM}=K,
\end{align*}
which  is equivalent to
\begin{equation}\label{e5}
\cos^4\theta\Phi^{2} \Re( A \sigma_1 B) \leqslant K^{2}\Phi^{2}\Re (A\sigma_2 B).
\end{equation}

Now, if $ p\geqslant 2, $ by Lemmas \ref{2},  \ref{l2} and inequality (\ref{m1}), respectively
\begin{align*}
\Vert \cos^p\theta&\Phi^{\frac{p}{2}}\Re(A \sigma_1 B) M^{\frac{p}{2}}m^{\frac{p}{2}}\Phi^{-\frac{p}{2}}\Re( A \sigma_2 B) \Vert \\
& \leqslant \frac{1}{4} \left \Vert \cos^p\theta\Phi^{\frac{p}{2}}\Re(A \sigma_1 B)+ M^{\frac{p}{2}}m^{\frac{p}{2}}\Phi^{-\frac{p}{2}}\Re( A \sigma_2 B) \right \Vert^{2}\\
 & \leqslant   \frac{1}{4} \left \Vert\cos^2\theta \Phi\Re(A \sigma_1 B)+ Mm\Phi^{-1}\Re( A \sigma_2 B) \right \Vert^{p} \\
 & \leqslant \frac{1}{4} (M+m)^{p}. \
\end{align*}
Hence,
$$\Vert \cos^p\theta\Phi^{\frac{p}{2}}\Re(A \sigma_1 B) \Phi^{-\frac{p}{2}}\Re( A \sigma_2 B) \Vert \leqslant  \left (\dfrac{(M+m)^{2}}{4^{\frac{2}{p}}Mm}  \right )^{\frac{p}{2}}.$$
This means that
$$\cos^{2p}\theta \Phi^{p}\Re(A \sigma_1 B) \leqslant  \left( \dfrac{(M+m)^{2}}{4^{\frac{2}{p}}Mm}  \right)^{p}\Phi^{p}\Re( A \sigma_2 B).$$
\end{proof}
This result was proved in \cite{bed1} for $p=2$ with coefficient $\sec^{12}\theta$ and we find a sharper inequality with coefficient $\sec^4\theta$.

Furthermore, if in Theorem \ref{t2}, we put $p=2$, $\sigma_1=\nabla_v$ and $\sigma_2=\sharp_v$, we get a refinement of some results in \cite{nas2} and \cite{Yan2}.

Going back to the proof of Proposition \ref{pro1}, if we put $\sigma_1=\nabla_v$, then $\Re A\sigma_1\Re B=\Re(A\sigma_1 B)$. Thus the coefficient $\cos\theta$ does not appear. Therefore, we have the following result.
\begin{theorem}\label{t2'}
Let $A,B\in {\mathcal S}_{\theta}$ such that $0<mI\leqslant \Re A,\Re B\leqslant MI$ and $\Phi$ be a positive unital linear map. If  $!_v\leqslant\sigma\leqslant\nabla_v$, then
\begin{align*}
&\Re^p\Phi(A\nabla_v B)\leqslant \alpha ^p \Re ^p\Phi(A\sigma B),
\end{align*}
where $\alpha=\max\{K,4^{1-\frac{2}{p}}K\}$.
\end{theorem}
In addition, \eqref{1}, helps us to prove another version of Theorem \ref{t2}.

 \begin{theorem}\label{t5}
Let $A,B\in {\mathcal S}_{\theta}$ such that $0<mI\leqslant \Re A^{-1},\Re B^{-1}\leqslant MI$ and $\sigma$ be an arbitrary mean. Then for every positive unital linear map $\Phi$ and $\sigma^*\leqslant\sigma_1,\sigma_2\leqslant\sigma$,
\begin{equation*}
\Phi^p\Re( A\sigma_1 B)\leqslant \alpha^p\sec^{4p}\theta \Phi^p\Re( A\sigma_2 B),
\end{equation*}
where $\alpha=\max\{K,4^{1-\frac{2}{p}}K\}$.
\end{theorem}
\begin{proof}
It is enough to prove
\begin{equation}\label{e01}
Mm\Phi\Re (A\sigma_1 B)+\Phi^{-1}\Re(A\sigma_2 B)\leqslant\sec^2\theta\ (M+m)I.\end{equation}
The rest of proof is similar to Theorem \ref{t2}.

As in the proof of Proposition \ref{pro1},
$$\Re A^{-1}\sigma\Re B^{-1}+mM(\Re^{-1} A^{-1}\sigma\Re^{-1} B^{-1})\leqslant (M+m)I.$$
From this inequality and   \eqref{1}, it follows that
$$\Re A^{-1}\sigma\Re B^{-1}+Mm(\Re A\sigma\Re B)\leqslant (M+m)I.$$
Lemma \ref{2} leads to
\begin{equation}\label{e02}\Re (A\sigma_1 B)\leqslant\sec^2\theta (\Re A\sigma_1 \Re B)\leqslant\sec^2\theta (\Re A\sigma\Re B). \end{equation}
In addition,
\begin{align}\label{e03}
\notag \Phi^{-1}\Re(A\sigma_2 B)&\leqslant \Phi^{-1}(\Re A\sigma_2 \Re B)\leqslant \Phi^{-1}(\Re A\sigma^* \Re B)\\\notag
&\leqslant \Phi(\Re A\sigma^* \Re B)^{-1}=\Phi(\Re^{-1} A\sigma \Re^{-1} B)\\
&\leqslant \sec^2\theta\Phi(\Re A^{-1}\sigma \Re B^{-1} ).
\end{align}
Combining \eqref{e02} and \eqref{e03}, we get \eqref{e01}.
\end{proof}
\begin{remark}
If we substitute $A$ , $B$, $\sigma$ with $A^{-1}$, $B^{-1}$ and $\sigma^*$ in this theorem, we get
$$\Phi^p\Re( A\sigma_1 B)^{-1}\leqslant \alpha^p\sec^{4p}\theta \Phi^p\Re( A\sigma_2 B)^{-1},$$
where  $0<mI\leqslant \Re A,\Re B\leqslant MI$.
\end{remark}
\begin{lemma}\label{9}\cite{bed1}
Let $f\in \mathbf{m}$ and $ A,B\in {\mathcal S_{\theta}}$ for some $0 \leqslant \theta <\frac{\pi}{2}$. Then
\begin{equation*}
f(\Re A)\leqslant\Re(f(A))\leqslant\sec^2\theta f(\Re A).
\end{equation*}
\end{lemma}
 \begin{lemma}\label{2.10}\cite[Proposition 2.1]{Mal}
If $\sigma_1$ and $\sigma_2$ are two means with $\sigma_1\leqslant\sigma_2$ and $A, B\in\mathcal S_{\theta}$, then
\begin{align*}
&(a)\quad\Re (A\sigma_1 B)\leqslant \sec^2\theta\ \Re (A\sigma_2 B)\\
&(b)\quad\Re(A\sigma_2 B)^{-1}\leqslant \sec^2\theta\ \Re (A\sigma_1 B)^{-1}.
\end{align*}
\end{lemma}

The following theorem is an analogue of Lemma \ref{2.10} involving operator monotone functions.
\begin{theorem}\label{t02}
Let $A,B\in {\mathcal S }_{\theta}$ such that $0<mI\leqslant \Re A,\Re B\leqslant MI$, $\sigma$ be an arbitrary mean, $\Phi$ be a positive unital linear map and $f\in \mathbf{m}$. Then
\begin{equation*}
\Re f\left(\Phi(A\sigma_1 B)\right)\leqslant \sec^{2}\theta\ \Re f\left(K\sec^{2}\theta  \ \Phi(A\sigma_2 B)\right)\leqslant K\sec^{4}\theta\ \Re f\left(\Phi(A\sigma_2 B)\right)
\end{equation*}
for every operator means $\sigma_1,\sigma_2$ between $\sigma$ and $\sigma^*$.

Moreover, if $\sigma_1\leqslant \sigma_2$, the constant $K$ can be omitted.
\end{theorem}
\begin{proof} From Lemma  \ref{9} and Theorem \ref{t2}, we have
\begin{align*}
\cos^{2}\theta \Re f\left(\Phi(A\sigma_1 B)\right)&\leqslant f\left(\Re\Phi(A\sigma_1 B)\right)\\
&\leqslant f\left(K\sec^{2}\theta\  \Re \Phi(A\sigma_2 B)\right)\\
&\leqslant \Re f\left(K\sec^{2}\theta  \Phi(A\sigma_2 B)\right).
\end{align*}
This prove the first inequality.

The second inequality follows from the fact that if $\lambda\geq1$ then $f(\lambda x)\leqslant \lambda f(x)$.

If $\sigma_1\leqslant\sigma_2$, we can use Lemma \ref{2.10} instead of Theorem \ref{t2} and avoid the constant $K$.
\end{proof}
Similarly using Theorem \ref{t2'}, we have this Theorem.
\begin{theorem}
Let $A,B\in {\mathcal S }_{\theta}$ such that $0<mI\leqslant \Re A,\Re B\leqslant MI$,  $\Phi$ be a positive unital linear map and $f\in \mathbf{m}$. Then
\begin{equation*}
\Re f\left(\Phi(A\nabla_v B)\right)\leqslant \sec^{2}\theta\ \Re f\left(K \Phi(A\sigma B)\right)\leqslant K\sec^{2}\theta\ \Re f\left(\Phi(A\sigma B)\right)
\end{equation*}
where $!_v\leqslant \sigma\leqslant\nabla_v$.
\end{theorem}

The next theorem is the sectorial version of \eqref{13}.
\begin{theorem}\label{t4}
Let $A,B\in {\mathcal S}_{\theta}$ such that $0<mI\leqslant \Re A,\Re B\leqslant MI$. Then for every positive unital linear map $\Phi$ and $\sigma_1$ and $\sigma_2$ between $\nabla_{\nu}$ and $!_{\nu}$,
\begin{equation*}
\Re\left( f(\Phi(A))\sigma_1f(\Phi( B))\right)\leqslant K\sec^{4}\theta\ \Re f\left(\Phi(A\sigma_2 B)\right),
\end{equation*}
for $f\in \mathbf{m}$.
\end{theorem}
\begin{proof} By Lemmas \ref{2} and \ref{9},
\begin{align*}
\Re\left( f(\Phi(A))\sigma_1f(\Phi( B))\right)
&\leqslant\sec^2\theta\ \left(\Re f(\Phi(A))\sigma_1\Re f(\Phi( B))\right)&&\\
&\leqslant \sec^4\theta f(\Re (\Phi(A)))\sigma_1f(\Re (\Phi(B)))&&\\
&=\sec^4\theta f( \Phi(\Re(A)))\sigma_1 f(\Phi(\Re (B)))\end{align*}
Now applying \eqref{13} and Lemmas  \ref{2} and \ref{9}, we get
\begin{align*}
 f( \Phi(\Re(A)))\sigma_1 f(\Phi(\Re (B)))&\leqslant K f\left(\Phi(\Re A\sigma_2\Re B)\right)&&\\
&\leqslant K\ f\left(\Phi(\Re( A\sigma_2 B))\right)&&\\
&=K f\left(\Re(\Phi( A\sigma_2 B))\right)\\
&\leqslant K \Re f(\Phi( A\sigma_2 B)).&&
\end{align*}
\end{proof}
\begin{remark}
In this Theorem, consider $\sigma_2=\nabla_v$. Since $f\in\mathbf{m}$ is operator concave, the last part of the proof can be written as
\begin{align*}
 f( \Phi(\Re A ))\sigma_1 f(\Phi(\Re  B ))&\leqslant f( \Phi(\Re A ))\nabla_v f(\Phi(\Re  B ))\\&\leqslant  f\left(\Phi(\Re A) \nabla_v\Phi(\Re B)\right)&&\\&= f\left(\Phi(\Re A \nabla_v\Re B)\right)&&\\
&= f\left(\Re(\Phi( A\nabla_v B))\right)\\
&\leqslant  \Re f(\Phi( A\nabla_v B)).&&
\end{align*}
This leads to
\begin{equation*}
\Re\left( f(\Phi(A))\sigma_1f(\Phi( B))\right)\leqslant \sec^{4}\theta\ \Re f\left(\Phi(A\nabla_v B)\right),
\end{equation*}
This result was proved in \cite{bed1} where $\sigma_1$ is geometric mean.
\end{remark}

\section{determinant, singularvalue and norm inequalities}
First, we intend to use Theorems \ref{t2}, \ref{t2'} and \ref{t5} to state determinant, singularvalue and norm versions of these theorems.

For determinant version, we use the following lemma.
\begin{lemma}\label{67}(\cite{Hor}, \cite{Lin1})
If $A\in {\mathcal S}_{\theta}$, then
\begin{equation*}
\det(\Re A)\leqslant \vert \det( A)\vert\leqslant \sec^n\theta\ \det(\Re A).
\end{equation*}
\end{lemma}
\begin{theorem}
Let $A,B\in {\mathcal S}_{\theta}$ and $\sigma^*\leqslant\sigma_1,\sigma_2\leqslant\sigma$  for some operator mean $\sigma$.
\begin{itemize}
\item  If $0<mI\leqslant \Re A,\Re B\leqslant MI$, then
\begin{equation*}
\mid\det(A\sigma_1 B)\mid\leqslant\sec^{3n}\theta\ K^n\mid\det(A\sigma_2B)\mid.
\end{equation*}
\item If $0<mI\leqslant \Re A^{-1},\Re B^{-1}\leqslant MI$, then
\begin{equation*}
\mid\det(A\sigma_1 B)\mid\leqslant\sec^{5n}\theta\ K^n\mid\det(A\sigma_2B)\mid.
\end{equation*}
\end{itemize}
In addition, if $!_v\leqslant\sigma\leqslant\nabla_v$ and $0<mI\leqslant \Re A,\Re B\leqslant MI$, then
\begin{equation*}
\mid\det(A\nabla_v B)\mid\leqslant\sec^{n}\theta\ K^n\mid\det(A\sigma B)\mid.
\end{equation*}
\end{theorem}
The following lemma, leads us to singularvalue version.
\begin{lemma}\label{31}(\cite{Bha1},\cite{Dur})
Let $A\in {\mathcal S}_{\theta}$. Then
\begin{equation*}
\lambda_j(\Re A)\leqslant s_j(A)\leqslant \sec^2\theta\lambda_j(\Re A),~~~~j=1,...,n.
\end{equation*}
\end{lemma}
\begin{theorem}
Let $A,B\in {\mathcal S}_{\theta}$ and $\sigma^*\leqslant\sigma_1,\sigma_2\leqslant\sigma$  for some operator mean $\sigma$.
\begin{itemize}
\item  If $0<mI\leqslant \Re A,\Re B\leqslant MI$, then
\begin{equation*}
s_j(A\sigma_1 B)\leqslant\sec^{4}\theta\ Ks_j(A\sigma_2B).
\end{equation*}
\item If $0<mI\leqslant \Re A^{-1},\Re B^{-1}\leqslant MI$, then
\begin{equation*}
s_j(A\sigma_1 B)\leqslant\sec^{6}\theta\ Ks_j(A\sigma_2B).
\end{equation*}
\end{itemize}
In addition, if $!_v\leqslant\sigma\leqslant\nabla_v$ and $0<mI\leqslant \Re A,\Re B\leqslant MI$, then
\begin{equation*}
s_j(A\nabla_v B)\leqslant\sec^{2}\theta\ Ks_j(A\sigma B).
\end{equation*}
\end{theorem}

Finally, this lemma, helps us to prove the norm versions.
 \begin{lemma}\label{32}(\cite{Bha1},\cite{zha3})
Let $A\in {\mathcal S}_{\theta}$. Then
\begin{equation*}
\|\Re A\|\leqslant \| A\|\leqslant \sec\theta\| \Re A\|
\end{equation*}
for any unitarily invariant norm $\|\cdot\|$ on $B(H)$.
\end{lemma}
\begin{theorem}
Let $A,B\in {\mathcal S}_{\theta}$ and $\sigma^*\leqslant\sigma_1,\sigma_2\leqslant\sigma$  for some operator mean $\sigma$.
\begin{itemize}
\item  If $0<mI\leqslant \Re A,\Re B\leqslant MI$, then
\begin{equation*}
\|A\sigma_1 B\|\leqslant\sec^{3}\theta\ K\|A\sigma_2B\|.
\end{equation*}
\item If $0<mI\leqslant \Re A^{-1},\Re B^{-1}\leqslant MI$, then
\begin{equation*}
\|A\sigma_1 B\|\leqslant\sec^{5}\theta\ K\|A\sigma_2B\|.
\end{equation*}
\end{itemize}
In addition, if $!_v\leqslant\sigma\leqslant\nabla_v$ and $0<mI\leqslant \Re A,\Re B\leqslant MI$, then
\begin{equation*}
\|A\nabla_v B\|\leqslant\sec\theta\ K\|A\sigma B\|.
\end{equation*}
\end{theorem}
At the end of this section, we prove similar determinant and norm inequalities for $A\sigma B$.
\begin{lemma}\label{80}(\cite{Yan3})
If $A,B\in {\mathcal S}_{\theta}$, then
\begin{equation}\label{80}
\mid\det(A+ B)\mid\leqslant\sec^{2n}\theta\ \mid\det(I_n+A)\mid.\mid\det(I_n+B)\mid.
\end{equation}
and
\begin{equation}\label{81}
\|A+B\|\leqslant\sec\ \theta\ \|I_n+A\|.\|I_n+B\|.
\end{equation}
\end{lemma}
\begin{theorem}
Let $A,B\in S_{\theta}$ and $\sigma$ be an arbitrary mean such that $\sigma\leqslant\nabla$. Then
\begin{equation*}
\mid\det(A\sigma B)\mid\leqslant\frac{\sec^{4n}\theta}{2^n} \mid\det(I_n+A)\mid.\mid\det(I_n+B)\mid.
\end{equation*}
and
\begin{equation*}
\|A\sigma B\|\leqslant\frac{\sec^{4}\theta}{2} \|I_n+A\|.\|I_n+B\|.
\end{equation*}
\end{theorem}
\begin{proof}
Compute
\begin{align*}
\mid\det(A\sigma B)\mid&\leqslant\sec^{n}\theta\det(\Re(A\sigma B))&&(\text{by Lemma  \ref{67} })\\
&\leqslant\frac{\sec^{3n}\theta}{2^n}\det(\Re(A+B))&&(\text{by Lemma \ref{2.10} })\\
&\leqslant\frac{\sec^{3n}\theta}{2^n}\mid\det(\Re(A+B)\mid&&(\text{by Lemma  \ref{67} })\\
&\leqslant\frac{\sec^{4n}\theta}{2^n} \mid\det(I_n+A)\mid.\mid\det(I_n+B)\mid.&&(\text{by   \eqref{80} })
\end{align*}
and
\begin{align*}
\|A\sigma B\|&\leqslant\sec\theta\|\Re(A\sigma B))\|&&(\text{by Lemma  \ref{32}})\\
&\leqslant\frac{\sec^{3}\theta}{2}\|\Re(A+B)\|&&(\text{by Lemma \ref{2.10}})\\
&\leqslant\frac{\sec^{3}\theta}{2}\|A+B\|&&(\text{by Lemma  \ref{32} })\\
&\leqslant\frac{\sec^{4}\theta}{2} \|I_n+A\|.\|I_n+B\|.&&(\text{by   \eqref{81} })
\end{align*}
\end{proof}

\section{Numerical range of sector matrices}
The numerical radius $\omega(A)$ of $A\in \Bbb M_n$ is defined by
\[
\omega(A)=\sup\{\langle Ax,x\rangle : x\in \mathbb{C}^n, \|x\|=1\}.
\]
When $A\in S_0$, we have $\omega(A)=\|A\|$. Therefore
\begin{equation}\label{3.6}
\omega(\Re A)=\|\Re A\|.
\end{equation}

Bedrani et al. \cite{bed1} showed if $ A,B\in {{\mathcal S}_{\theta}}$, $\nu\in[0,1]$ and $f\in \mathbf{m}$, then

 \begin{equation}\label{3.2}
f(\|\Re A\|)\leqslant\|\Re(f(A))\|\leqslant\sec^2\theta f(\|\Re A\|),
\end{equation}
\begin{equation}\label{3.1}
\omega(\Re A)\leqslant \omega(A)\leqslant \sec\theta\ \omega(\Re A).
\end{equation}

\begin{theorem}\label{t7}
Let $A, B\in\mathcal S_{\theta}$ such that $0<mI\leqslant\Re A,\Re B\leqslant MI$. If $f\in \mathbf{m}$, and $\sigma^*\leqslant\sigma_1,\sigma_2\leqslant\sigma$ for some operator mean $\sigma$, then
\begin{equation*}
f(\omega(A\sigma_1 B))\leqslant\sec^3\theta\ K\omega\left(f(A\sigma_2 B)\right).
\end{equation*}
\begin{proof}
Using Theorem $\ref{t2}$, concavity of $f$, \eqref{3.2} and \eqref{3.1}, we have
\begin{align*}
 f\left(\omega(A\sigma_1 B)\right)&\leqslant f\left(\sec\theta\ \omega(\Re(A\sigma_1 B))\right)\\
&\leqslant f\left(\sec\theta\ \omega(\sec^2\theta K\Re( A\sigma_2 B))\right)\\
&= f\left(\sec^3\theta K\omega(\Re(A\sigma_2 B)\right)\\
&\leqslant\sec^3\theta K f\left(\omega(\Re(A\sigma_2 B)\right)\\
&= \sec^3\theta Kf\left(\|\Re(A\sigma_2 B)\|\right)\\
&\leqslant\sec^3\theta K \|\Re f(A\sigma_2 B)\|\\
&=\sec^3\theta K\omega\left(\Re(f(A\sigma_2 B)\right)\\
&\leqslant\sec^3\theta K\omega\left(f(A\sigma_2 B)\right),
\end{align*}
which completes the proof.
\end{proof}
\end{theorem}
\begin{remark}
If we use Theorem \ref{t2'} instead of Theorem \ref{t2}, we have
\begin{equation*}
f(\omega(A\nabla_v B))\leqslant\sec\theta\ K\omega\left(f(A\sigma B)\right),
\end{equation*}
for $!_v\leqslant\sigma\leqslant\nabla_v$.
\end{remark}
The next proposition is reverse of \cite[Theorem 3.5]{bed2}.
\begin{proposition}\label{p4.1}
Let $A,B\in {\mathcal S}_{\theta}$ such that $0<mI\leqslant \Re A,\Re B\leqslant MI$. Then for every positive unital linear map $\Phi$ and $\sigma_1$ and $\sigma_2$ between $\nabla_{\nu}$ and $!_{\nu}$,
\begin{equation*}
\omega\left(f(\Phi(A))\sigma_1f(\Phi(B))\right)\leqslant K\sec^5\theta\ \omega\left( f(\Phi(A\sigma_2B))\right).
\end{equation*}
for $f\in \mathbf{m}$.
\begin{proof}
From Theorem $\ref{t4}$ and \eqref{3.1}, we have
\begin{align*}
\omega\left(f(\Phi(A))\sigma_1f(\Phi(B))\right)&\leqslant\sec\theta\ \omega\left(\Re f(\Phi(A))\sigma_1f(\Phi(B))\right)\\
&\leqslant K\sec^5\theta\ \omega\left(\Re f(\Phi(A\sigma_2B))\right)\\
&\leqslant K\sec^5\theta\ \omega\left( f(\Phi(A\sigma_2B))\right).
\end{align*}
\end{proof}
\end{proposition}

\section*{Statements \& Declarations}
The authors declare that no funds, grants, or other support were received during the preparation of this manuscript.

The authors have no relevant financial or non-financial interests to disclose.

All authors contributed to the design and implementation of the research, to the analysis of the 
results and to the writing of the manuscript. All authors read and approved the final manuscript.


\end{document}